\documentclass[11pt,a4paper]{article}
\usepackage[english]{babel}
\usepackage{times}
\usepackage{graphicx}
\usepackage{amscd}
\usepackage{amsmath}
\usepackage{amsfonts}
\usepackage{amssymb}
\usepackage{amsthm}
\usepackage{latexsym}
\usepackage{bigints}
\newtheorem{theorem}{Theorem}

\newtheorem{lemma}[theorem]{Lemma}

\newtheorem{remark}[theorem]{Remark}

\def\neweq#1{\begin{equation}\label{#1}}
\def\endeq{\end{equation}}

\begin{document}

\title{Variational formulation of the Melan equation}

\author{Filippo GAZZOLA \ -- \ Yongda WANG \ -- \ Raffaella PAVANI\\
{\small Dipartimento di Matematica, Politecnico di Milano (Italy)}
}

\date{}
\maketitle
\begin{abstract}
The Melan beam equation modeling suspension bridges is considered. A slightly modified equation is derived by applying variational principles and by
minimising the total energy of the bridge. The equation is nonlinear and nonlocal, while the beam is hinged at the endpoints. We show that the problem
always admits at least one solution whereas the uniqueness remains open although some numerical results suggest that it should hold. We also emphasize the
qualitative difference with some simplified models.\par
\end{abstract}

\section{Introduction}

At the end of the 19th century, Josef Melan \cite{melan} suggested the following fourth order ordinary differential equation
to describe the behavior of suspension bridges:
\begin{equation}\label{melan}
EI w''''(x)-(H+h(w))w''(x)-h(w)y''(x)=p\qquad \forall x\in (0,L),
\end{equation}
where $L$ is the distance between the two towers, $w=w(x)$ denotes the vertical displacement of the beam representing the deck, $y=y(x)$ is the position
of the sustaining cable at rest, $E$ and $I$ are, respectively, the elastic modulus of the material composing the deck and the moment of inertia of the
cross section so that $EI$ is the flexural rigidity, $H$ is the horizontal tension of the cable when subject to the dead load $q=q(x)$, $h(w)$ represents the
additional tension in the cable produced by the live load $p=p(x)$. The dead load $q$ includes the weights of the cable, of the hangers, and of the deck.
In the book by von K\'{a}rm\'{a}n-Biot \cite[(5.5)]{karbio}, \eqref{melan} is called the {\em fundamental equation of the theory of the suspension bridge}. In Figure \ref{bridge} we sketch a picture of a suspension bridge.

From a mathematical point of view, the additional tension $h(w)$ in \eqref{melan} deserves a particular attention since it is nonlocal and it introduces
a nonlinearity into the equation. As we shall see, the computation of $h(w)$ is delicate and, in literature, there are several different ways to approximate
it, see \cite{gjs,karbio,ty}. For both the original term $h(w)$ and these approximate forms of it, one can show that there exists at least one solution
of the Melan equation \eqref{melan} with hinged boundary conditions, see \cite[Section 5]{gjs}.
The Melan equation \eqref{melan} is also quite challenging for numerical analysts, see
\cite{dl,gp,los,se,se1,se2,woll} where several approximating procedures for the solution of \eqref{melan} have been discussed for different forms of the term $h(w)$.

The purpose of the present paper is to derive the Melan equation from a variational principle and to study its behavior. We prove that the Euler-Lagrange
equation, see \eqref{melanp}, admits at least one solution and we prove uniqueness for certain values of the parameters. We also discuss uniqueness
for the remaining values of the parameters and we give some numerical results which show how delicate and unstable the equation is. Finally,
we emphasize the role of the nonlinearity and a qualitative difference between the solution of the variational problem with the solution of a simplified problem.

\section{How to derive the Melan equation from a variational principle}

Following von K\'arm\'an-Biot \cite[Section VII.5]{karbio}, we view the main span of a suspension bridge as a combined system of a perfectly flexible string (the
sustaining cable) and a beam (the deck). The beam and the string are connected by a large number of inextensible hangers, see Figure \ref{bridge}.
\begin{figure}
\centering
\includegraphics[height=30mm, width=100mm]{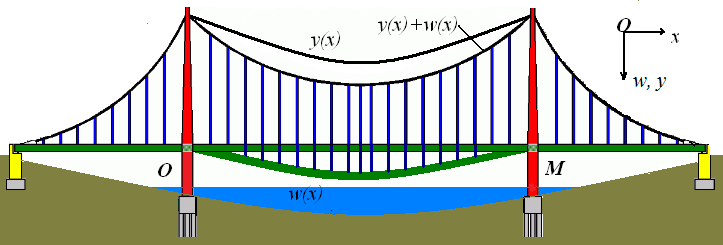}
\caption{Beam sustained by a cable through parallel hangers.\label{bridge}}
\end{figure}
The point $O$ is the origin of the orthogonal coordinate system and positive displacements are oriented downwards. The point $M$ has coordinates
$(0,L)$ with $L$ being the length of the deck between the two towers.

If no live loads act on the beam, there is no bending moment in the beam and the cable is in the position $y(x)$, while the unloaded beam is in the horizontal
position of the segment connecting $O$ and $M$. In this situation, the horizontal component $H>0$ of the tension of the cable is constant. Hence, there is an
equilibrium position in the system and the configuration of the cable is obtained by solving the equation (see \cite[(1.3), Section VII]{karbio})
\begin{equation}\label{cableequ}
H y''(x)=-q\qquad \forall x\in (0,L).
\end{equation}
Since the endpoints of the cable are at the same level $\ell$ (the height of the towers) and since the dead load $q$ is constant, the solution of \eqref{cableequ} is
\[y(x)=\ell+\frac{q}{2H}x(L-x)\qquad\forall x\in (0,L):\]
hence, the cable has the shape of a parabola and
\begin{equation}\label{yprime}
y'(x)=\frac{q}{H}\left(\frac{L}{2}-x\right),\quad y''(x)=-\frac{q}{H}\qquad \forall x\in (0,L).
\end{equation}
Therefore, the length of the cable at rest is
\begin{equation}\label{lc}
L_c=\int_0^L{\sqrt{1+y'(x)^2}}dx=\frac{L}{2}\sqrt{1+\frac{q^2L^2}{4H^2}}+\frac{H}{q}\log\left(\frac{q L}{2H}+\sqrt{1+\frac{q^2L^2}{4H^2}}\right).
\end{equation}

When a live load $p$ acts on the deck of the bridge, the beam may leave the horizontal position and produce a displacement $w$
(positive displacements are oriented downwards). By assuming that $w$, $w'$, $w''$ are small, let us compute the energies involved when the
system is in this new position.

The potential energy produced by the live and dead loads reads
\[\mathcal{E}_L=\int_0^L{(p+q)w}dx.\]

The elastic energy $\mathcal{E}_{B}$ necessary to bend the beam is the squared curvature times half the flexural rigidity, that is,
\[
\mathcal{E}_{B}=\frac{EI}{2}\int_0^L{\frac{(w'')^2}{(1+(w')^2)^3}\sqrt{1+(w')^2}}dx=\frac{EI}{2}\int_0^L{\frac{(w'')^2}{(1+(w')^2)^{5/2}}}dx\approx \frac{EI}{2}\int_0^L{(w'')^2}dx,
\]
where we used the fact that $w'$ and $w''$ are both small and we neglect the terms $o(w')^2$ of superquadratic order.

Since the cable is assumed to be perfectly flexible, it has no resistance to bending. Then the only internal force is its tension which consists of two
parts: the tension at rest $H(x)$ and the additional tension $h(w)$ due to the variation of length of the cable. The former is
$H(x)=H\sqrt{1+y'(x)^2}$ and the amount of energy needed to deform the cable at rest under the tension $H(x)$ in the infinitesimal interval $[x,x+dx]$
from the original position $y(x)$ to the new position $y(x)+w(x)$ is the variation of length times the tension, that is,
\begin{align*}
\mathcal{E}_{C_1}dx=H\sqrt{1+(y')^2}\left(\sqrt{1+(w'+y')^2}-\sqrt{1+(y')^2}\right)dx.
\end{align*}
Then, the energy necessary to deform the whole cable at rest under the tension $H(x)$ is
\begin{align*}
\mathcal{E}_{C_1}&=H \int_0^L{\sqrt{1+(y')^2}\left(\sqrt{1+(w'+y')^2}-\sqrt{1+(y')^2}\right)}dx\\
&=H \int_0^L{\left[1+(y')^2\right]\left(\sqrt{1+\frac{2w'y'}{1+(y')^2}+\frac{(w')^2}{1+(y')^2}}-1\right)}dx.
\end{align*}
In order to maintain only the at most quadratic terms, we use the asymptotic expansion $\sqrt{1+\varepsilon}\approx 1+\frac{\varepsilon}{2}-\frac{\varepsilon^2}{8}$ as $\varepsilon \to 0$; with an integration by parts we then get
\begin{align}\label{ec1}
\mathcal{E}_{C_1}&\approx H\int_0^L{\left[1+(y')^2\right]\left(\frac{w'y'}{1+(y')^2}+\frac{(w')^2}{2(1+(y')^2)^2}\right)}dx=\frac{H}{2} \int_0^L{\frac{(w')^2}{1+(y')^2}}dx+H\int_0^L{w'y'}dx\nonumber\\
&= \frac{H}{2} \int_0^L{\frac{(w')^2}{1+(y')^2}}dx+q\int_0^L{w}dx.
\end{align}

For the additional tension $h(w)$, we note that the hangers connecting the cable and the beam are inextensible. Therefore, the deflection of the cable follows the displacement $w$ of the deck, that is, the cable reaches the new position $y+w$ and the variation of the cable length due to the deformation $w$ is given by (see \eqref{yprime}-\eqref{lc})
\begin{equation}\label{gamma}
\Gamma(w)=\int_0^L{\sqrt{1+\left(w'(x)+y'(x)\right)^2}}dx-L_c.
\end{equation}
If $A$ denotes the cross-sectional area of the cable and $E_c$ is the modulus of elasticity, then the additional tension $h(w)$ in the cable produced by the live load $p$ and the corresponding energy $\mathcal{E}_{C_2}(w)$ are given by
\begin{equation}\label{addten}
h(w)=\frac{E_cA}{L_c}\Gamma(w),\qquad \mathcal{E}_{C_2}(w)=\frac{E_cA}{2L_c}\Gamma(w)^2.
\end{equation}

Therefore, the total energy necessary to deform the cable is
\[\mathcal{E}_C=\mathcal{E}_{C_1}+\mathcal{E}_{C_2}=\frac{H}{2} \int_0^L{\frac{(w')^2}{1+(y')^2}}dx+q\int_0^L{w}dx+\frac{E_cA}{2L_c}\Gamma(w)^2.\]

\begin{remark}\label{re}
When computing the energy $\mathcal{E}_{C_1}$, Timoshenko-Young obtain
\[\mathcal{E}_{C_1}=\frac{H}{2} \int_0^L{(w')^2}dx+q\int_0^L{w}dx,\]
see \cite[Section 11.16]{ty}. This formula should be compared with \eqref{ec1}: it is obtained by approximating $y'\approx 0$. As explained in \cite{gjs} this may generate some significant errors in the solutions. It was the civil and structural German engineer Franz Dischinger who discovered around 1950 the dramatic consequences of
this approximation on the structures.
\end{remark}

Summarizing, the total energy in the system after the deformation $w$ is
\begin{align*}
\mathcal{E}=\mathcal{E}_B+\mathcal{E}_{C}-\mathcal{E}_L=
\frac{EI}{2}\int_0^L{(w'')^2}dx+\frac{H}{2} \int_0^L{\frac{(w')^2}{1+(y')^2}}dx+\frac{E_cA}{2L_c}\Gamma(w)^2-\int_0^L{pw}dx.
\end{align*}
The Euler-Lagrange equation of the system is obtained by taking the critical points of the energy $\mathcal{E}$. Then by recalling \eqref{yprime} and that the beam is hinged at its endpoints, we obtain the following boundary value problem
\begin{equation}\label{melanp}
\begin{cases}
EI w''''(x)-H \left(\frac{w'(x)}{1+(y'(x))^2}\right)'-\frac{E_c A}{L_c}\frac{w''(x)-q/H}{(1+(w'(x)+y'(x))^2)^{3/2}}\Gamma(w)=p\quad & x\in (0,L)\\
w(0)=w(L)=w''(0)=w''(L)=0.\quad &
\end{cases}
\end{equation}

One should compare \eqref{melanp} with the classical Melan equation \eqref{melan}. For the history and the details on the derivation of the Melan equation we also refer to the recent monograph \cite{ga}.

\section{Main result}

For simplicity, we put $a=EI$, $b=H$, and $c=\frac{E_cA}{L_c}$. Then the problem \eqref{melanp} reads
\begin{equation}\label{melanp1}
\begin{cases}
a w''''(x)- b\left(\frac{w'(x)}{1+(y'(x))^2}\right)'-c\frac{w''(x)-q/H}{(1+(w'(x)+y'(x))^2)^{3/2}}\Gamma(w)=p\quad & x\in (0,L)\\
w(0)=w(L)=w''(0)=w''(L)=0,\quad &
\end{cases}
\end{equation}
where $a, b, c>0$ and the functional $\Gamma(w)$ is as in \eqref{gamma}, it is nonlinear nonlocal and of indefinite sign. Define $\alpha, \beta>0$ by
\begin{equation}\label{ab}
\alpha^2:=\left[1+\frac{q^2L^2}{12H^2}\right]\frac{L}{H},\qquad \beta^2:=\left[1+\frac{q^2L^2}{4H^2}\right]\frac{1}{H}.
\end{equation}
Given $k\in [1, \infty]$, we denote the $L^k$-norm by $\|u\|_k$ for any $u \in L^k(0,L)$.
We also introduce the following scalar product on the second order Sobolev space $H^2\cap H_0^1(0,L)$:
\begin{equation}\label{newsc}
(u,v)_{y}:=a\int_0^L{u''v''}dx+b\int_0^L{\frac{u'v'}{1+(y')^2}}dx\quad \mbox{ for any }u,v\in H^2\cap H_0^1(0,L),
\end{equation}
where the function $y'$ is as in \eqref{yprime}. Let $\mathcal{H}$ be the dual space of $H^2\cap H_0^1(0,L)$; we denote by $\|\cdot\|_{\mathcal{H}}$ the $\mathcal{H}$-norm and by $\langle\cdot,\cdot\rangle$ the corresponding duality between $H^2\cap H_0^1(0,L)$ and $\mathcal{H}$.

If $p\in \mathcal{H}$, we say that $w\in H^2\cap H_0^1(0,L)$ is a weak solution of \eqref{melanp1} if
\begin{equation}\label{define}
(w,v)_{y}+c\Gamma(w)\int_0^L{\frac{(w'+y')v'}{\sqrt{1+(w'+y')^2}}}dx=\langle p,v\rangle\qquad \mbox{for all } v \in H^2\cap H_0^1(0,L).
\end{equation}
Then we prove

\begin{theorem}\label{thm}
For any $p\in \mathcal{H}$, there exists at least one weak solution of the problem \eqref{melanp1}. Moreover, assume that
\begin{equation}\label{cc}
0<c<\frac{1}{\alpha^{2}}.
\end{equation}
Then for all $p\in \mathcal{H}$ satisfying
\begin{equation}\label{pc}
\|p\|_{\mathcal{H}}<\frac{(1-c\alpha^2)^2}{c\alpha\beta^2},
\end{equation}
the problem \eqref{melanp1} admits a unique weak solution $w\in H^2\cap H_0^1(0,L)$.
\end{theorem}

The uniqueness statement holds if both $c>0$ and $\|p\|_{\mathcal{H}}$ are sufficiently small. However, the assumption $c<\alpha^{-2}$ does not hold in
general for actual bridges, see \cite{woll}. Therefore, we now discuss the case $c\geq\alpha^{-2}$.\par
First, we study what happens in the ``limit case'' where $c\to +\infty$: the problem \eqref{melanp1} degenerates to
\begin{equation}\label{degp}
\frac{w''-q/H}{(1+(w'+y')^2)^{3/2}}\Gamma(w)=0\quad x\in (0,L)\ ,\qquad w(0)=w(L)=0.
\end{equation}
Clearly, $w=0$ and $w=-\frac{q}{2H}(L-x)x$ are two solutions of \eqref{degp}. But let us also analyze the functional $\Gamma(w)$. We shift it by
$Y(x):=\frac{q}{2H}x(L-x)$ (so that $Y\in H^2\cap H_0^1(0,L)$ and $Y'=y'$) and,
for all $0\not \equiv w\in H^2\cap H_0^1(0,L)$, we define the real function
$$
\gamma_w(t):=\Gamma(tw-Y)=\int_0^L{\left[\sqrt{1+(tw')^2}-\sqrt{1+(y')^2}\right]}dx\qquad \mbox{ for any }t \in \mathbb{R}.
$$
Clearly, $\gamma_w(\pm\infty)=+\infty$ and $\gamma_w$ is strictly convex in $\mathbb{R}$. Since $\gamma_w(0)<0$, there exist $T^-_w<0<T^+_w$
such that $\gamma_w(T^\pm_w)=0$. Hence, for any $w\neq0$ we have $\Gamma(T^\pm_ww\!-\!Y)=0$, that is, $T^\pm_ww\!-\!Y$ solves \eqref{degp};
therefore, \eqref{degp} admits infinitely many solutions.
The qualitative graph of the functional $\frac{c}{2}\Gamma^2$ is depicted in Figure \ref{gwp}; since $J_p-\frac{c}{2}\Gamma^2$
is convex (see again Figure \ref{gwp}), if $c$ is large then the behavior of the functional $J_p$ is not clear; in this situation, the uniqueness and/or
multiplicity for \eqref{melanp1} is an open problem.
\begin{figure}
\centering
\includegraphics[height=20mm, width=130mm]{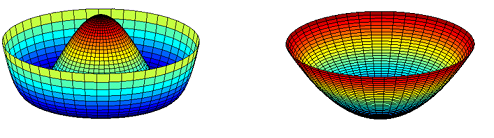}
\caption{Qualitative shape of the graphs of the functionals $w\mapsto \frac{c}{2}\Gamma(w)^2$ and $w\mapsto J_p(w)-\frac{c}{2}\Gamma(w)^2$.\label{gwp}}
\end{figure}

Second, we point out that the numerical results reported in the next section suggest that uniqueness holds also for $c\geq\alpha^{-2}$.

\section{Numerical results}

We consider \eqref{melanp1} in order to simulate the behavior of the real
three span suspension bridge already studied by Wollmann \cite{woll}. Using
his parameter values and his physical assumptions, we reduce our
computations to the main span which is assumed to be $460m$ long; then we have $L=460m$%
, $EI=57\times 10^{6}kN\cdot m^{2}$, $E_{c}A=36\times 10^{6}kN$, $q=170kN/m$%
, $H=97.75\times 10^{3}kN$, $\frac{q}{H}=1.739\times 10^{-3}m^{-1}$. More,
according to \cite{gjs}, we have $L_{c}=1.026\times
460 m=471.96m$. At last, we scale the length by $\gamma =\frac{1}{460}$ so
that the length becomes $L^{\ast}=1$: we call the new variable $s=\gamma x$ and we have $w(x)=z(\gamma x)=z(s)$. After some
computations we obtain
\begin{align}\label{numequ}
z''''(s)&-3.6289\times10^2\frac{z''(s)}{1+0.64(0.5-s)^{2}}-4.6442\times 10^{2}\frac{z'(s)(0.5-s)}{\left[
1+0.64(0.5-s)^{2}\right]^{2}}  \nonumber  \\
& -\frac{2.8318\times 10^{2}z''(s)-1.042\times 10^5}{(1+[2.1739\times
10^{-3}z'(s)+0.8(0.5-s)]^2)^{3/2}}\Gamma (z)=7.8555\times 10^{2}p(s),\qquad s\in
(0,1),
\end{align}%
where $\Gamma (z)=460\left[ \int_{0}^{1}\sqrt{1+\left[ 2.1739\times
10^{-3}z^{\prime }(s)+0.8(0.5-s)\right] ^{2}}ds-1.026\right]$ is the length increment of the cable, see \eqref{gamma}, and is measured in meters.

Assuming that a \textit{uniform live load} $p$ acts over the main span $L^{\ast}$, we solved \eqref{numequ} for many initial values of $\Gamma(z)$ by using the \textit{bvptwp} code, whose MATLAB version was published by Cash et al. \cite{ch}. It is an optimized high-quality code for the numerical solution of two-point
boundary value problems, which employs a mesh selection strategy based on
the estimation of the local error. In practice, a variable stepsize is
used in order to obtain a solution with a relative error less than the required
tolerance. We chose such tolerance $tol=10^{-6}$. It is worth noticing that
\textit{bvptwp} requires that the problem is posed as a first-order system.
For our computations, this is a great advantage, since it allows to have the
discrete first derivative $z'(s)$ with the same accuracy of the
solution $z(s)$. The main characteristics of the used code are that it implements high
order methods using a deferred correction strategy and often works
extremely efficiently on very difficult problems. Instead, the iterative
method presented by Dang-Luan \cite{dl} solves the boundary value problem
by a difference method of a second order convergence on uniform grid
and then estimates the derivative of solution by finite difference
approximations. We remark that a very good approximation of discrete
derivative values are required in order to compute $\Gamma(z)$ without
increasing the global error. The algorithm by Dang-Luan \cite{dl} does not
seem to hit this target.
\begin{figure}
\centering
\includegraphics[height=40mm, width=114mm]{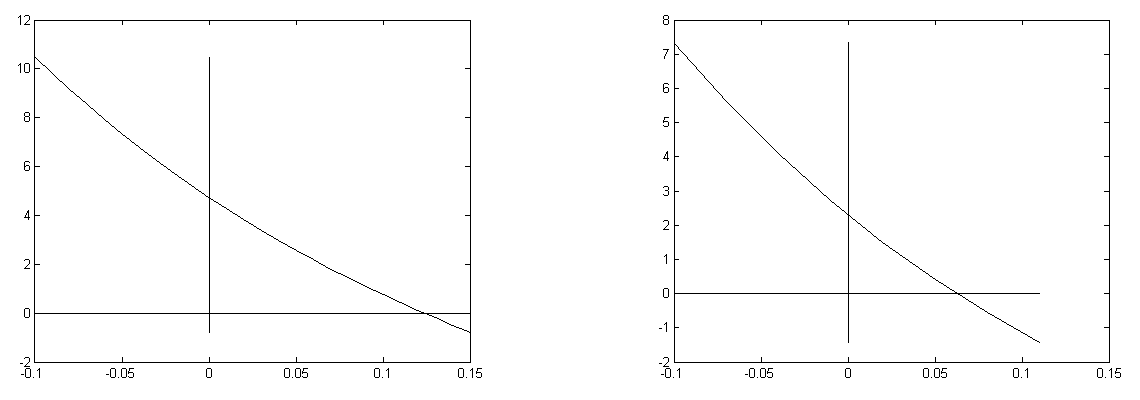}
\caption{The graph of the map $\Gamma_{in}\mapsto \Gamma_{out}$ with a uniform live load on the whole beam (left) and on the left half of the beam (right).\label{gwp1}}
\end{figure}

For each input $\Gamma _{in}$, an output $\Gamma _{out}$ is computed as follows: we solve \eqref{numequ} with $\Gamma (z)=\Gamma _{in}
$ and we find  $z_{out}$ and $z'_{out}$ on a discrete mesh with stepsize
$5\times 10^{-4}$, then we compute $\Gamma_{out}:=\Gamma (z_{out})$ by means of the composite trapezoidal rule, which is of the second
order, so we have an integration error which does not affect the global error. The map $\Gamma _{in}\mapsto \Gamma_{out}$ is plotted in Figure \ref{gwp1}
which shows that there exists a unique numerically unstable fixed point. This behavior remains the same for all the many values of $p$ we used.
Therefore we empirically conclude that a unique solution of \eqref{numequ} exists. Then we refined our computations and found value for which we have
$\Gamma _{in}\approx \Gamma _{out}$ that we consider the required fixed point $\Gamma_{fix}$. For instance, we found the fixed
points for different loads $p$ as shown in Table \ref{fix}.
\begin{table}
\centering
\begin{tabular}{|l|l|l|l|l|l|}
\cline{1-4} \cline{6-6}
\multicolumn{1}{|l|}{$p$ ${\scriptscriptstyle(kN/m)}$} & \multicolumn{1}{c|}{10} & \multicolumn{1}{c|}{30} & \multicolumn{1}{c|}{60} &  &\multicolumn{1}{c|}{30$\cdot\chi_{\{(0,230)\}}$}\\ \cline{1-4}\cline{6-6}
\multicolumn{1}{|l|}{$\Gamma_{fix}$ ${\scriptscriptstyle(m)}$} & \multicolumn{1}{c|}{0.039977} & \multicolumn{1}{c|}{0.11966} & \multicolumn{1}{c|}{0.23843} &  &\multicolumn{1}{c|}{0.060441}\\ \cline{1-4}\cline{6-6}
\multicolumn{1}{|l|}{$M$ ${\scriptscriptstyle(m)}$} & \multicolumn{1}{c|}{0.06887} & \multicolumn{1}{c|}{0.2014} & \multicolumn{1}{c|}{0.3924} &  &\multicolumn{1}{c|}{0.9843} \\ \cline{1-4}\cline{6-6}
\multicolumn{1}{|l|}{$G$ ${\scriptscriptstyle(m)}$} & \multicolumn{1}{c|}{0.02} & \multicolumn{1}{c|}{0.05} & \multicolumn{1}{c|}{0.08} & & \multicolumn{1}{c|}{1.81}\\ \cline{1-4}\cline{6-6}
\end{tabular}
\caption{The fixed points, maximal displacements and gaps under different loads $p$.}\label{fix}
\end{table}
It is worth noticing that since the problem appears numerically unstable, at
least five significant digits have to be used. In all our experiments the map $\Gamma _{in}\mapsto \Gamma_{out}$ turned out to be strictly decreasing with very negative slope. This suggests the uniqueness
of $\Gamma_{fix}$ and, in turn, the uniqueness of the solution of \eqref{numequ}.
\begin{figure}
\centering
\includegraphics[height=38mm, width=111mm]{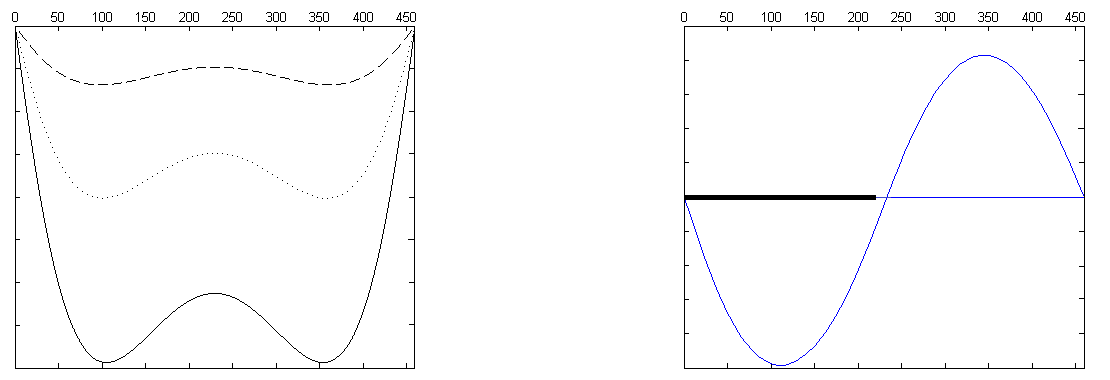}
\caption{Solutions of \eqref{numequ} with different loads $p$.\label{gwp2}}
\end{figure}

Using $\Gamma _{fix},$ we computed the numerical solutions of (\ref{numequ}). In the left picture in Figure \ref{gwp2}, using the original scale for
the $x$-axis, we report the solutions obtained for $p=10kN/m$ (dashed line), $p=30kN/m$ (dotted line), and $p=60kN/m$ (solid line).
We also computed the maximal displacements $M$ and the gap $G$ between the first maximum and the central relative minimum, see Table \ref{fix}.\par
Then we studied a live load having a weight density of $30kN/m$ and located on the left half of the span (e.g.\ a train of length $230m$).
In this case, the fixed point was computed to be $\Gamma_{fix}=0.060441m$ and we obtained the solution of \eqref{melanp1}, which is an ``almost" skew-symmetric
function, see the right one in Figure \ref{gwp2}.
Compared with the case where a uniform load $p=30 KN/m$ is on the whole beam, in the case where a uniform load $p=30 KN/m$ only acts on the left half of the beam, the shape of the cable changes (now it is not symmetrical with respect to $x=\frac{L}{2}$) and it results in a smaller increment length of the cable (see $\Gamma_{fix}$ in Table \ref{fix}). This produces a smaller upwards force (due to the extension of the cable) acting on the beam, and hence, the maximal displacement is larger than that under a uniform load $p$ on the whole beam, see Table \ref{fix}.\par
Overall, it is clear that the qualitative behavior of the solutions simulated by our nonlinear model \eqref{melanp1} is not close to those of
the linear model, which exhibits a unique maximum. Hence our results do not agree with the results reported by Semper \cite{se1}, referring to his nonlinear
model. Indeed our solutions show that the downwards deflection (assumed
downwards positive) presents two (equal) maxima and one relative minimum in between, instead of a unique maximum at
the middle span, as reported by Semper \cite{se1}. Indeed, Semper finds a small but not negligible discrepancy between the linear and the nonlinear equations, but for him this implies a quantitative difference only, whereas we find a significant qualitative difference also.\par
Finally, we considered the equation (\ref{numequ}) by dropping all the denominators and the third term (involving the first order derivative). The resulting
equations reads
$$
z''''(s)-3.6289\times10^2\cdot z''(s)
-\big(2.8318\times 10^{2}z''(s)-1.042\times 10^5\big)\Gamma (z)=7.8555\times 10^{2}p(s),\qquad s\in(0,1).
$$
The fixed point was $\Gamma_{fix}=0.212346$ instead of 0.11966 as for (\ref{numequ}) whereas the solution had a unique maximum computed equal to
$0.4175$. This striking difference gives strength to Remark \ref{re}, that is, one cannot neglect the curvature of the main cable.

\section{Proof of Theorem \ref{thm}}\label{4}

Since $y'$ is bounded and $a,b>0$, the scalar product \eqref{newsc} induces a norm on $H^2\cap H_0^1(0,L)$ denoted by
\begin{equation*}
\|u\|_{y}:=\left(a\int_0^L{(u'')^2}dx+b\int_0^L{\frac{(u')^2}{1+(y')^2}}dx\right)^{1/2}\quad \mbox{ for any }u\in H^2\cap H_0^1(0,L),
\end{equation*}
which is equivalent to $\|u''\|_2$, that is, the standard $H^2\cap H_0^1$-norm. For all $u\in H^2\cap H_0^1(0,L)$ we have
\begin{align}\label{ests}
\begin{array}{rcl}
\|u'\|_1^2 &\le& \int_0^L{(1+(y')^2)}dx\int_0^L{\frac{(u')^2}{1+(y')^2}}dx=\left[1+\frac{q^2L^2}{12H^2}\right]L\int_0^L{\frac{(u')^2}{1+(y')^2}}dx\le\alpha^2\|u\|_y^2\\
\|u'\|_2^2 &\le& \|1+(y')^2\|_{\infty}\int_0^L{\frac{(u')^2}{1+(y')^2}}dx=\left[1+\frac{q^2L^2}{4H^2}\right]\int_0^L{\frac{(u')^2}{1+(y')^2}}dx\le\beta^2\|u\|_y^2
\end{array}
\end{align}
with $\alpha, \beta>0$ as in \eqref{ab}. In addition, the simple inequality
\begin{equation}\label{inequ2}
\left|\sqrt{1+(\lambda+\mu)^2}-\sqrt{1+\mu^2}\right|\leq |\lambda|\qquad \forall \lambda, \mu\in \mathbb{R}
\end{equation}
implies that for any $u,v\in H_0^1\cap H^2(0,L)$
\begin{equation}\label{inequ3}
|\Gamma(u)|\leq \|u'\|_1,\qquad |\Gamma(u)-\Gamma(v)|\leq \|u'-v'\|_1.
\end{equation}

We first state the continuity and differentiability of the functional $\Gamma(w)^2$.

\begin{lemma}\label{gamma2}
Let $\Gamma(w)$ be as in \eqref{gamma}. Then $\Gamma(w)^2$ is weakly continuous and differentiable in $ H^2\cap H_0^1(0,L)$.
\end{lemma}

The proof of Lemma \ref{gamma2} is standard and we omit it. The energy functional corresponding to  \eqref{melanp1} is
\[J_p=J_p(w)=\frac{1}{2}\|w\|_{y}^2+\frac{c}{2}\Gamma(w)^2-\langle p,w\rangle\qquad \mbox{ for any }w\in H^2\cap H_0^1(0,L).\]
According to Lemma \ref{gamma2}, weak solutions of \eqref{melanp1} are the critical points of the functional $J_p$. The next step is to prove the geometrical properties (coercivity) and compactness properties (Palais-Smale (PS) condition) of $J_p$.

\begin{lemma}\label{ps}
For any $p\in \mathcal{H}$, the functional $J_p$ is coercive and bounded below in $H^2\cap H_0^1(0,L)$. Moreover, it satisfies the (PS) condition.
\end{lemma}

\begin{proof} Since $p\in \mathcal{H}$ and $c>0$, we have for any $w\in H^2\cap H_0^1(0,L)$
\[J_p\geq \frac{1}{2}\|w\|_{y}^2-\langle p,w\rangle\geq \frac{1}{2}\|w\|_{y}^2-\|p\|_{\mathcal{H}}\|w\|_{y}\geq-\frac{\|p\|_{\mathcal{H}}^2}{2},\]
which implies that the functional $J_p$ is coercive and bounded below.

Consider now a sequence $\{w_n\}$ such that $J_p(w_n)$ is bounded and $J'_p(w_n)\to 0$ in $\mathcal{H}$. Then
\begin{align*}
\exists M>0, \qquad M\geq  \frac{1}{2}\|w_n\|_{y}^2+\frac{c}{2}\Gamma(w_n)^2-\langle p,w_n\rangle
\geq  \frac{1}{2}\|w_n\|_{y}^2-\|p\|_{\mathcal{H}}\|w_n\|_{y}.
\end{align*}
Hence, $\|w_n\|_{y}$ is bounded and there exists some $\overline{w}\in H^2\cap H_0^1(0,L)$ such that $w_n$ weakly converges to $\overline{w}$ in $H^2\cap H_0^1(0,L)$, up to a subsequence. Therefore, one has $\langle J'_p(w_n),v\rangle\to \langle J'_p(\overline{w}),v\rangle$ for all $v \in H^2\cap H_0^1(0,L)$, which proves that $J'_p(\overline{w})=0$. It follows that
\begin{align}\label{jprime}\begin{array}{rcl}
\langle J'_p(w_n),w_n\rangle &=& \|w_n\|_{y}^2+c\Gamma(w_n) \langle \Gamma'(w_n),w_n\rangle -\langle p,w_n\rangle\\
&\to& 0=\langle J'_p(\overline{w}), \overline{w}\rangle=\|\overline{w}\|_{y}^2+c\Gamma(\overline{w})\langle \Gamma'(\overline{w}),\overline{w}\rangle-
\langle p,\overline{w}\rangle.
\end{array}\end{align}
By \eqref{inequ3} and compact embedding we know that $\Gamma(w_n) \langle \Gamma'(w_n),w_n\rangle\to\Gamma(\overline{w})\langle \Gamma'(\overline{w}),\overline{w}\rangle$. Since $\langle p,w_n\rangle\to\langle p,\overline{w}\rangle$, by \eqref{jprime} we deduce that $\|w_n\|_{y}\to \|\overline{w}\|_{y}$. Together with the weak convergence $w_n\rightharpoonup \overline{w}$, this shows that
$w_n\to\overline{w}$  in $H^2\cap H_0^1(0,L)$.
This proves (PS) condition.
\end{proof}

By Lemma \ref{ps}, the functional $J_p$ admits a global minimum in $H^2\cap H_0^1(0,L)$ for any $p\in \mathcal{H}$. This
minimum point is a critical point for $J_p$ and hence a weak solution of \eqref{melanp1}. This proves the first
part of Theorem \ref{thm}.

We now discuss uniqueness. We first remark that if $w$ is a weak solution of \eqref{melanp1}, then by \eqref{define} we have
\begin{equation}\label{equ}
\|w\|_y^2+c\Gamma(w)\int_0^L{\frac{(w'+y')w'}{\sqrt{1+(w'+y')^2}}}dx=\langle p, w\rangle\leq \|p\|_{\mathcal{H}}\|w\|_y.
\end{equation}
By \eqref{ests} and \eqref{inequ3} we deduce that
\[\left|\Gamma(w)\right|\left|\int_0^L{\frac{(w'+y')w'}{\sqrt{1+(w'+y')^2}}}dx\right|\leq\|w'\|_1^2\leq \alpha^2\|w\|_y^2\]
so that, assuming \eqref{cc}, from \eqref{equ} we infer the following a priori bound for solutions of \eqref{melanp1}:
\begin{equation}\label{ss}
\|w\|_y\leq (1-c\alpha^2)^{-1}\|p\|_{\mathcal{H}}:=R_p.
\end{equation}

Next we fix $v\in H^2\cap H_0^1(0,L)$ and consider the linear problem
\begin{equation}\label{linp}
\begin{cases}
a w''''(x)-b \left(\frac{w'(x)}{1+(y'(x))^2}\right)'=c\frac{v''(x)-q/H}{(1+(v'(x)+y'(x))^2)^{3/2}}\Gamma(v)+p\quad & x\in (0,L)\\
w(0)=w(L)=w''(0)=w''(L)=0,\quad &
\end{cases}
\end{equation}
where $0<c<\alpha^{-2}$. Since $\frac{v''(x)-q/H}{(1+(v'(x)+y'(x))^2)^{3/2}}\Gamma(v)\in \mathcal{H}$ and $p\in \mathcal{H}$, there exists a unique solution $w\in H^2\cap H_0^1(0,L)$ of \eqref{linp} due to the Lax-Milgram theorem. We define the closed ball $B_p$ and the map $\Phi$ by
\[B_p:=\{w\in H^2\cap H_0^1(0,L); \|w\|_y\leq R_p\},\qquad\Phi: B_p\to H^2\cap H_0^1(0,L);\quad \Phi(v)=w,\]
 with $w$  being the unique solution of \eqref{linp}.

\begin{lemma}\label{map}
If \eqref{cc} and \eqref{pc} hold, then the map $\Phi$ satisfies $\Phi(B_p)\subseteq B_p$ and it is contractive in $B_p$.
\end{lemma}

\begin{proof} For any fixed $v\in B_p$, by testing \eqref{linp} with its solution $w=\Phi(v)$, we get
\begin{align*}
\|w\|_y^2&=c\Gamma(v)\int_0^L{\frac{(v''-q/H)w}{(1+(v'+y')^2)^{3/2}}}dx+\langle p,w\rangle\\
&\leq -c\Gamma(v)\int_0^L{\frac{(v'+y')w'}{\sqrt{1+(v'+y')^2}}}dx+\|p\|_{\mathcal{H}}\|w\|_y\\
\mbox{ by \eqref{ests}-\eqref{inequ3}}&\leq c\|v'\|_1\|w'\|_1+\|p\|_{\mathcal{H}}\|w\|_y
\leq \left(c\alpha^2\|v\|_y+\|p\|_{\mathcal{H}}\right)\|w\|_y\\
&\leq \left(c\alpha^2 R_p+\|p\|_{\mathcal{H}}\right)\|w\|_y=R_p\|w\|_y.
\end{align*}
Hence, $\|w\|_y\leq R$ which shows that $\Phi(B_p)\subseteq B_p$.

Note that the function $s\mapsto s/\sqrt{1+s^2}$ is globally Lipschitzian with constant 1, that is,
\begin{equation}\label{23}
\left|\frac{s_1}{\sqrt{1+s_1^2}}-\frac{s_2}{\sqrt{1+s_2^2}}\right|\leq |s_1-s_2|\qquad \forall s_1,s_2\in \mathbb{R}.
\end{equation}

Take $v_1,v_2\in B_p$ and let $w_1=\Phi(v_1)$, $w_2=\Phi(v_2)$, then we have for all $u\in H^2\cap H_0^1(0,L)$
\[(w_i,u)_y=c\Gamma(v_i)\int_0^L{\frac{(v''_i-q/H)u}{(1+(v'_i+y')^2)^{3/2}}}dx+\langle p,u\rangle\qquad i=1,2.\]
Put $u=w=w_1-w_2$, subtract these two equations  and recall \eqref{yprime}. Then, after integration by parts we get
\begin{align*}
\|w\|_y^2=&c\Gamma(v_2)\int_0^L{\frac{(v'_2+y')w'}{\sqrt{1+(v'_2+y')^2}}}dx-c\Gamma(v_1)\int_0^L{\frac{(v'_1+y')w'}{\sqrt{1+(v'_1+y')^2}}}dx\\
=&c\Gamma(v_2)\left[\int_0^L{\frac{(v'_2+y')w'}{\sqrt{1+(v'_2+y')^2}}}dx-\int_0^L{\frac{(v'_1+y')w'}{\sqrt{1+(v'_1+y')^2}}}dx\right]\\
&+c\left[\Gamma(v_2)-\Gamma(v_1)\right]\int_0^L{\frac{(v'_1+y')w'}{\sqrt{1+(v'_1+y')^2}}}dx\\
\mbox{ by \eqref{inequ3}-\eqref{23}  } &\leq c\|v'_2\|_1\int_0^L{|v'_1-v'_2| |w'|}dx+c\|v'_1-v'_2\|_1\|w'\|_1\\
\mbox{ by \eqref{ests}-\eqref{inequ2}  }&\leq  c \alpha \beta^2R_p\|v_1-v_2\|_y\|w\|_y+c \alpha^2\|v_1-v_2\|_y\|w\|_y.
\end{align*}

Hence, by the definition of $B_p$ in \eqref{ss}, we infer that
\[\|\Phi(v_1)-\Phi(v_2)\|_y=\|w\|_y\leq c \alpha\left(\alpha+ \beta^2 R_p\right)\|v_1-v_2\|_y= c
\alpha\left(\alpha+\frac{ \beta^2 \|p\|_{\mathcal{H}}}{1-c\alpha^2}\right) \|v_1-v_2\|_y:=\rho\|v_1-v_2\|_y.\]
Since the condition \eqref{pc} yields that $0<\rho<1$, this proves that $\Phi$ is contractive in $B_p$.
\end{proof}

Assume \eqref{cc} and \eqref{pc}. From \eqref{ss} we know that any solution of \eqref{melanp1} belongs to $B_p$. By Lemma \ref{map} and the Banach Contraction principle, $\Phi$ admits a unique fixed point in $B_p$, which solves \eqref{melanp1}. This completes the proof of Theorem \ref{thm}.

\section{Conclusion}

We considered a variational form of the Melan equation, see \eqref{melanp}. The novelty consists in taking into account the
shape of the cable and not replacing $y'(x)$ with 0, as erroneously done in \cite{karbio}. Indeed, von K\'arm\'an-Biot \cite[p.277]{karbio} warn the reader
by writing that {\em whereas the deflection of the beam may be considered small, the deflection of the string, i.e., the deviation of its shape from a straight
line, has to be considered as of finite magnitude}; then, they {\em neglect $y'(x)^2$ in comparison with unity}, see \cite[(5.14)]{karbio}.
We also maintained the nonlinearity given by the additional tension in the sustaining cable. This gives some difficulties in proving uniqueness of the
solution, see Theorem \ref{thm} and the comments that follow. Our numerical results suggest that one may have uniqueness for any $c>0$ but Figure \ref{gwp}
leaves some doubts. The numerical procedure turns out to be extremely unstable, see the plots in Figure \ref{gwp1}. The numerically found solutions exhibit a clear
nonlinear behavior of the equations and a strong dependence on the curvature of the cable. Hence, one cannot drop the nonlinearity nor approximate
$1+(y')^2\approx1$. We are confident that this paper might be the starting point for refined theoretical and numerical researches on the Melan equation.

\end{document}